\documentclass{amsart}
\usepackage[draft]{say}

\usepackage{graphicx}
\usepackage{verbatim}
\usepackage{amsmath}
\usepackage{amssymb}
\usepackage{amsthm}
\usepackage{ifthen}
\usepackage{amsfonts}
\usepackage{mathtools}
\usepackage{enumerate}
\usepackage{caption}
\usepackage{subcaption}
\usepackage{hyperref}
\hypersetup{colorlinks=true,citecolor=blue,linkcolor=blue}

\newtheorem{proposition}{Proposition}[section]
\newtheorem{theorem}[proposition]{Theorem}

\newtheorem{lemma}[proposition]{Lemma}
\newtheorem{prop}[proposition]{Proposition}
\newtheorem{cor}[proposition]{Corollary}
\newtheorem{conj}[proposition]{Conjecture}
\newtheorem{obs}{Observation}

\theoremstyle{definition}

\theoremstyle{remark}

\numberwithin{equation}{section}

\newcommand{\newword}[1]{\textbf{\emph{#1}}}

\newcommand{\integers}{\mathbb Z}
\newcommand{\reals}{\mathbb R}
\newcommand{\complexes}{\mathbb C}

\newcommand{\set}[1]{{\lbrace #1 \rbrace}}

\renewcommand{\c}{{\mathbf c}}
\renewcommand{\d}{{\mathbf d}}
\newcommand{\g}{{\mathbf g}}

\newcommand{\m}{{\mathbf m}}

\renewcommand{\th}{^\mathrm{th}}

\newcommand{\A}{\mathcal{A}}
\newcommand{\T}{\mathbb{T}}
\newcommand{\dashname}[1]{\stackrel{#1}{\begin{picture}(22,3)\put(0,2.5){\line(1,0){22}}\end{picture}}}
\newcommand{\ck}{^{\vee\!}}
\newcommand{\sgn}{\operatorname{sgn}}

\begin{document}

\title[Initial-seed recursions and dualities for $\d$-vectors]{Initial-seed recursions and dualities\\ for $\d$-vectors}
\author{Nathan Reading and Salvatore Stella}
\address[Nathan Reading]{North Carolina State University, Raleigh, NC, USA}
\email{reading@math.ncsu.edu}
\address[Salvatore Stella]{Department of Mathematics \& Department of Computer Science, University of Haifa, Haifa, Mount Carmel 31905, Israel}
\email{stella@mat.uniroma1.it}
\thanks{Nathan Reading was partially supported by NSF grant DMS-1101568.}

\begin{abstract}
  We present an initial-seed-mutation formula for $\d$-vectors of cluster variables in a cluster algebra.
  We also give two rephrasings of this recursion: one as a duality formula for $\d$-vectors in the style of the $\g$-vectors/$\c$-vectors dualities of Nakanishi and Zelevinsky, and one as a formula expressing the highest powers in the Laurent expansion of a cluster variable in terms of the $\d$-vectors of any cluster containing it.
	We prove that the initial-seed-mutation recursion holds in a varied collection of cluster algebras, but not in general.
  We conjecture further that the formula holds \emph{for source-sink moves on the initial seed} in an arbitrary cluster algebra, and we prove this conjecture in the case of surfaces.
\end{abstract}
\maketitle

\setcounter{tocdepth}{1}
\tableofcontents

\section{Introduction}\label{intro}

This paper concerns the search for an \newword{initial-seed recursion} for $\d$-vectors: a recursive formula for how $\d$-vectors change under mutation of initial seeds.
We begin this introduction by providing background on cluster algebras, seeds, and $\d$-vectors.

The origins of cluster algebras lie in the study of totally positive matrices, generalized by Lusztig \cite{Lusztig} to a notion of totally positive elements in any reductive group.
Indeed, the recursive definition of cluster algebras extends and generalizes a recursion on minimal sets of minors whose positivity implies total positivity of matrices.
Cluster algebras were introduced by Fomin and Zelevinsky \cite{dbc,ca1}, who conjectured that the coordinate ring of any \textit{double Bruhat cell} (i.e. any intersection of two Bruhat cells for opposite Borel subgroups) is a cluster algebra.
(As it turns out, the natural choice of cluster algebra is a subring of the double Bruhat cell, proper in some cases.
In general, the double Bruhat cell coincides with a related larger algebra called an \textit{upper cluster algebra} \cite{ca3}.)

Since their introduction, cluster algebras and/or their underlying combinatorics and geometry have been found in widely different settings.
Some of these settings---and some early references---are algebraic geometry (Grassmannians \cite{Scott} and tropical analogues \cite{SpWi}), discrete dynamical systems (rational recurrences \cite{CaSp,FZLP}), higher Teichm\"{u}ller theory \cite{FG1,FG2}, PDE (KP solitons \cite{KW1,KW2}), Poisson geometry \cite{GSV1,GSV2}, representation theory of quivers/finite dimensional algebras \cite{BMRRT,BMRT,CCS,CK,MRZ}, scattering diagrams \cite{GHK,GHKK,KS} (related to mirror symmetry, Donaldson-Thomas theory, and integrable systems, and string theory),  and $Y$-systems in thermodynamic Bethe Ansatz \cite{ga}.

We begin by reviewing the definition of a (coefficient free) \newword{cluster algebra}.
An \newword{exchange matrix} $B=(b_{ij})$ is a skew-symmetrizable $n\times n$ integer matrix (meaning that there exist positive integers $d_i$ such that  $d_ib_{ij} = -d_jb_{ji}$ for every $i$ and $j$).
We write $\T_n$ for the $n$-regular tree with edges properly labeled $1,\ldots,n$, and we distinguish one vertex $t_0$ as the ``initial'' vertex.
We will write $t\dashname{k} t'$ to indicate that $t$ and $t'$ are connected by an edge labeled $k$.
We define a function $t\mapsto B_t$ that labels each vertex of $\T_n$ with an exchange matrix.
Specifically, we set $B_{t_0}$ equal to some ``initial'' exchange matrix $B_0$ and, for each edge $t\dashname{k} t'$ with $B_t=(b_{ij})$, we insist that $B_{t'}=(b'_{ij})$ be given by
\begin{equation}
  \label{b mut}
  b_{ij}'=\left\lbrace\!\!\begin{array}{ll}
    -b_{ij}&\mbox{if }i=k\mbox{ or }j=k;\\
    b_{ij}+\sgn(b_{kj})\,[b_{ik}b_{kj}]_+&\mbox{otherwise.}
  \end{array}\right.
\end{equation}
Here and elsewhere in the text, the notation $[a]_+$ means $\max(a,0)$ while $\sgn(a)$ is the sign of $a$.

Taking $x_1,\ldots,x_n$ to be indeterminates, we also label each vertex $t$ of $\T_n$ with an $n$-tuple $(x_{1;t},\ldots,x_{n;t})$ of rational functions in $x_1,\ldots,x_n$ called \newword{cluster variables}.
The label on $t_0$ consists of the indeterminates: $x_{i;t_0}=x_i$ for all $i$.
The remaining cluster variables are prescribed by \newword{exchange relations}.
For each edge $t\dashname{k} t'$, we have $x_{i;t'}=x_{i;t}$ for all $i\neq k$ and
\begin{equation}
  \label{x mut}
  x_{k;t}x_{k;t'}=\prod_{i=1}^nx_{i;t}^{[b_{ik}]_+}+\prod_{i=1}^nx_{i;t}^{[-b_{ik}]_+},
\end{equation}
where the $b_{ik}$ are entries of $B_t$.

Each pair $(B_t,(x_{1;t},\ldots,x_{n;t}))$ is called a \newword{seed}.
When $t$ and $t'$ are connected by an edge $t\dashname{k} t'$, the relationship between the seeds $(B_t,(x_{1;t},\ldots,x_{n;t}))$ and $(B_{t'},(x_{1;t'},\ldots,x_{n;t'}))$ is called \newword{mutation in direction $k$}.
The (coefficient-free) cluster algebra $\A(B_0)$ associated to the initial exchange matrix $B_0$ is the algebra (a subalgebra of the field of rational functions in $x_1,\ldots,x_n$) generated by the set $\set{x_{i;t}:\,t\in\T_n,\,i=1,\ldots,n}$ of all cluster variables.
Typically, there are infinitely many cluster variables;  when the set $\set{x_{i;t}:\,t\in\T_n,\,i=1,\ldots,n}$ is finite, we say that $B_0$ is of \newword{finite type}.

The first fundamental result on cluster algebras is the \newword{Laurent Phenomenon} \cite[Theorem~3.1]{ca1}.
The exchange relations define the cluster variables as rational functions in $x_1,\ldots,x_n$.
The Laurent Phenomenon is the assertion that each cluster variable is in fact a Laurent polynomial (a polynomial divided by a monomial).
This implies in particular that each cluster variable has a \newword{denominator vector} or \newword{$\d$-vector}.
The $\d$-vector of $x_{i;t}$ is a vector $\d_{j;t}$ with $n$ entries, whose $j\th$ entry is the power of $x_j^{-1}$ that appears as a factor of $x_{i;t}$.
In principle, the $\d$-vector may have negative entries (when powers of $x_j$ appear in the numerator of $x_{i;t}$), but in practice this only happens when $x_{i;t}$ equals some $x_j$.

Denominator vectors are fundamental to the theory of cluster algebras in many ways, and they are also significant in other settings beginning with Fomin and Zeleivinsky's proof \cite{ga} of Zamolodchikov's periodicity conjecture on $Y$-systems in the theory of thermodynamic Bethe ansatz.
They are also important in representation theory.
Each skew-symmetric $n\times n$ exchange matrix $B$ defines a \textit{quiver} (i.e.\ a directed graph) $Q$ on the vertices $1,\ldots,n$.
(The signs of entries give the direction of arrows and the magnitudes of entries give multiplicities of arrows.)
In the case where $B$ is skew-symmetric and acyclic, the $\d$-vectors of cluster variables are exactly the dimension vectors of \textit{rigid} indecomposable modules over the path algebra of $Q$ (modules with no self-extensions).
(See \cite{BMRT,CCS}.)
In combinatorics, the $\d$-vectors, realized as \textit{almost positive roots} in an associated root system, are central to the structure of \textit{generalized associahedra} and thus play a role in Coxeter-Catalan combinatorics \cite{Armstrong,rsga} and are interesting in more general settings such as subword complexes, multiassociahedra, graph associahedra, and so forth.

Once we know the Laurent Phenomenon, the exchange relations \eqref{x mut} imply a recursion on $\d$-vectors $\d_{j;t}$, given later as \eqref{usual D}.
This recursion is a ``final-seed recursion'' because it describes how $\d$-vectors (computed with respect to a fixed \textit{initial} seed) change when we mutate the \textit{final} seed $(B_t,(x_{1;t},\ldots,x_{n;t}))$.

We are now prepared to discuss the search for an initial-seed recursion for $\d$-vectors, describing how $\d$-vectors at a fixed \emph{final} seed change under mutation of \emph{initial} seeds.
It is widely expected (see e.g.\ \cite[Remark~7.7]{ca4}) that no satisfactory initial-seed-mutation recursion holds in general, and indeed we do not produce one.
However, a very nice initial-seed-mutation recursion holds in a varied collection of cluster algebras (including the case considered in \cite[Remark~7.7]{ca4}).
This recursion turns out to be equivalent to a beautiful duality formula in the style of the $\g$-vectors/$\c$-vectors dualities of Nakanishi and Zelevinsky \cite{N,NZ}.

The first thing one notices when looking for such a recursion is that, to understand how denominators change when the initial seed is mutated, one must know something about a related family of integer vectors.
Specifically, if $(x_1,\ldots,x_n)$ is the initial cluster, then the \emph{negation} of the $\d$-vector of a cluster variable $x$ is the vector of lowest powers of the $x_i$ occurring in the expression for $x$ as a Laurent polynomial in $x_1,\ldots,x_n$.
We define the \newword{$\m$-vector} of $x$ to be the vector of \emph{highest} powers of the $x_i$ occurring in $x$.
Our initial-seed-mutation recursion for $\d$-vectors is equivalent to a description of the $\m$-vectors in a given cluster in terms of the $\d$-vectors in the same cluster.

In many cases, one can establish the three formulas \eqref{DD}--\eqref{MD} by reading off the duality directly from expressions for denominator vectors found in the literature \cite{CP,FST,LLZ}.
In particular, all of them hold in finite type, in rank two (i.e.\ $n=2$), and more intriguingly, in nontrivial examples arising from marked surfaces.

We conjecture that the initial-seed-mutation recursion holds in the case of source-sink moves in arbitrary cluster algebras.
We prove this conjecture for cluster algebras arising from surfaces.
Dylan Rupel \cite{Dylan} has proved the conjecture in the case where $B$ is acyclic, using a categorification of quantum cluster algebras.

Besides their usefulness in understanding denominator vectors, the $\m$-vectors may be of independent interest.
A major goal in the study of cluster algebras is to give explicit formulas for the cluster variables.
Work in this direction includes realizing cluster variables as ``lambda lengths'' in the surfaces case \cite{FT}, combinatorial formulas in rank two \cite{LS}, in some finite types \cite{Mus,Sch}, and for some surfaces \cite{MS,MSW,ST}, interpretations in terms of the representation theory of quivers, beginning with \cite{CC}, and formulas in terms of ``broken lines'' in scattering diagrams~\cite{GHKK}.
Short of a complete description of a cluster variable, one might instead describe its Newton polytope (the convex hull of the exponent vectors of the Laurent monomials occurring in its Laurent expansion).
However, as far as the authors are aware, there are no general results describing Newton polytopes.
(For a description in one finite-type case, see \cite{Kal}.)

Together, the $\d$-vectors and $\m$-vectors amount to coarse information about Newton polytopes, namely their ``bounding boxes.'' Given a polytope $P$ in $\reals^n$, define the \newword{tight bounding box} of $P$ to be the smallest box $[a_1,b_1]\times\cdots\times[a_n,b_n]$ containing $P$.
(Readers who pay attention to bounding boxes of graphics files will   find the notion familiar.)
Equivalently, for each $i=1,\ldots,n$, the values $a_i$ and $b_i$ are respectively the minimum and maximum of the $i\th$ coordinates of points in~$P$.
It is convenient to describe the tight bounding box by specifying the vectors $(a_1,\ldots,a_n)$ and $(b_1,\ldots,b_n)$.
The tight bounding box of the Newton polytope of a Laurent polynomial $f$ in $x_1,\ldots,x_n$ is $[a_1,b_1]\times\cdots\times[a_n,b_n]$ such that $a_i$ is the lowest power of $x_i$ occurring in any Laurent monomial of $f$, and $b_i$ is the highest power of $x_i$ occurring.
Thus when $x$ is a cluster variable written as a Laurent polynomial in the initial cluster $(x_1,\ldots,x_n)$, the tight bounding box of the Newton polytope of $x$ is given by the negation of the $\d$-vector and by the $\m$-vector.

\section{Results}
Our notation is in the spirit of \cite{ca4} and \cite{NZ}.
As before, the notation $[a]_+$ means $\max(a,0)$.
We will apply the operators $\max$, $|\,\cdot\,|$ , and $[\,\cdot\,]_+$ entry-wise to vectors and matrices.
We continue to write $\T_n$ for the $n$-regular tree with edges properly labeled $1,\ldots,n$.
Symbols like $t$, $t_0$, $t'$, etc.\ will stand for vertices of $\T_n$.
The notation $t\dashname{k} t'$ indicates an edge in $\T_n$ labeled $k$.
In what follows, the initial seed is allowed to vary, so we need to be able to indicate the initial seed as part of the notation.
Thus, the notation $B_t^{B_0;t_0}$ stands for the exchange matrix at $t$, where $B_0$ is the exchange matrix at $t_0$.
Similarly, $x_{j;t}^{B_0;t_0}$ stands for the (coefficient-free) cluster variable indexed by $j$ in the (labeled) seed at $t$, and $\d_{j;t}^{B_0;t_0}$ is the denominator vector of $x_{j;t}^{B_0;t_0}$ with respect to the cluster at~$t_0$.

Given a matrix $A$, let $A^{\bullet k}$ be the matrix obtained from $A$ by replacing all entries outside the $k\th$ column with zeros.
Similarly, $A^{k \bullet}$ is obtained by replacing entries outside the $k\th$ row with zeros.
Let $J_k$ be the matrix obtained from the identity matrix by replacing the $kk$-entry by~$-1$.
The superscript $T$ stands for transpose.

We fix $(x_1,\ldots,x_n)$ to be the initial cluster (the cluster at $t_0$).
We write $D_t^{B_0;t_0}$ for the matrix whose $j\th$ \emph{column} is $\d_{j;t}^{B_0;t_0}$ and $D_{ij;t}^{B_0;t_0}$ for the $ij$-entry of that matrix.
Each $x_{j;t}^{B_0;t_0}$ is a Laurent polynomial in $x_1,\ldots,x_n$.
(This is the Laurent Phenomenon, \cite[Theorem~3.1]{ca1}.)
Let $M_t^{B_0;t_0}$ be the matrix whose $ij$-entry $M_{ij;t}^{B_0;t_0}$ is the maximum, over all of the (Laurent) monomials in $x_{j;t}^{B_0;t_0}$, of the power of $x_i$ occurring in the monomial.
Write $\m_{j;t}^{B_0;t_0}$ for the $j\th$ \emph{column} of $M_t^{B_0;t_0}$ and call this the $j\th$ \newword{$\m$-vector} at $t$.

We now present a duality property for denominator vectors that holds in some cluster algebras, as well as two equivalent properties: an initial-seed-mutation recursion for denominator vectors and a formula for the $M$-matrix at a given seed in terms of the $D$-matrix at the same seed.

\medskip

\noindent
\textbf{Property D.} ($D$-matrix duality).
For vertices $t_0,t\in\T_n$, writing $B_t$ as shorthand for $B_t^{B_0;t_0}$,
\begin{equation}\label{DD}
  \bigl(D_t^{B_0;t_0}\bigr)^T=D_{t_0}^{(-B_t)^T;t}.\\
\end{equation}

\medskip

\noindent
\textbf{Property R.} (Initial-seed-mutation recursion for $D$-matrices).
Suppose $t_0\dashname{k}t_1$ is an edge in $\T_n$ and write $B_1$ for $\mu_k(B_0)$.
Then
\begin{equation}
  \label{D backwards}
  D_t^{B_1;t_1}= J_kD_t^{B_0;t_0}+\max\left([B_0^{k\bullet}]_+D_t^{B_0;t_0},\,\, [-B_0^{k\bullet}]_+D_t^{B_0;t_0} \right)\\
\end{equation}

The recursion in Property R is not on individual denominator vectors, but rather on an entire cluster of denominator vectors.
For $i\neq k$, the $i\th$ entry of each denominator vector is unchanged, while row $k$ of the $D$-matrix (the vector of $k\th$ entries in denominator vectors) transforms by a recursion similar to the usual recursion (equation \eqref{usual D}, below) for how denominator vectors change under mutation.

\medskip

\noindent
\textbf{Property M.} ($M$-matrices in terms of $D$-matrices).
For vertices $t_0,t\in\T_n$,
\begin{equation}
  \label{MD}
  M_t^{B_0;t_0}=-D_t^{B_0;t_0}+\max\left([B_0]_+D_t^{B_0;t_0},\,\, [-B_0]_+D_t^{B_0;t_0} \right)\\
\end{equation}

\medskip

When Property M holds, in particular, the entire tight bounding box of a cluster variable $x$ can be determined directly from the denominator vectors of any cluster containing $x$.

Our first main result is the following theorem, which we prove in Section~\ref{general sec}.

\begin{theorem}
  \label{DRM}
	Fix a (coefficient-free) cluster pattern $t\mapsto(B^{B_0;t_0}_t,(x_{1;t},\ldots,x_{n;t}))$.
  The following are equivalent:
  \begin{enumerate}
    \item Property D holds for all $t_0$ and $t$.
    \item Property R holds for all $t_0$, $t$, and $k$.
    \item Property M holds for all $t_0$ and $t$.
  \end{enumerate}
\end{theorem}

A natural question is to characterize the cluster algebras in which Properties D, R, and M hold.
As a start towards answering this question, we prove the following three theorems in Section~\ref{special sec}.
In every case, the proof is to read off Property D using a known formula for the denominator vectors.

\begin{theorem}
  \label{D rk 2}
  Properties D, R, and M holds in any cluster pattern whose exchange matrices are $2\times2$.
\end{theorem}

\begin{theorem}
  \label{D finite}
  Properties D, R, and M hold in any cluster pattern of finite type.
\end{theorem}

\begin{theorem}
  \label{D surface}
  Properties D, R, and M hold for a cluster algebra arising from a marked surface if and only if the marked surface is one of the following.
  \begin{enumerate}
    \item
      \label{disk good}
      A disk with at most one puncture (finite types A and D).
    \item
      \label{small Atilde good}
      An annulus with no punctures and one or two marked points on each boundary component (affine types $\tilde A_{1,1}$, $\tilde A_{2,1}$, and $\tilde A_{2,2}$).
    \item
      \label{small Dtilde good}
      A disk with two punctures and one or two marked points on the boundary component (affine types $\tilde D_3$ and $\tilde D_4$).
    \item
      \label{sphere 4 good}
      A sphere with four punctures and no boundary components.
    \item
      \label{torus 1 good}
      A torus with exactly one marked point (either one puncture or one boundary component containing one marked point).
  \end{enumerate}
\end{theorem}

In Section~\ref{general sec}, we also prove some easier relations on $D$-matrices and $M$-matrices that hold in general.
The first of these shows that, to understand how $D$-matrices transform under mutation of the initial seed, one must understand $M$-matrices.
\begin{prop}
  \label{MD init}
  Suppose $t_0\dashname{k}t_1$ is an edge in $\T_n$.
  Then $D_t^{B_1;t_1}$ is obtained by replacing the $k\th$ row of $D_t^{B_0;t_0}$ with the $k\th$ row of $M_t^{B_0;t_0}$.  That is,
  \[
    D_t^{B_1;t_1}=D_t^{B_0;t_0}-(D_t^{B_0;t_0})^{k\bullet}+(M_t^{B_0;t_0})^{k\bullet}.
  \]
\end{prop}

The final-seed mutation recursion on denominator vectors \cite[(7.6)--(7.7)]{ca4} is given in matrix form as follows.
The initial $D$-matrix $D_{t_0}^{B_0;t_0}$ is the negative of the identity matrix, and for each edge $t\dashname{k}t'$ in $\T_n$,
\begin{equation}
  \label{usual D}
  D_{t'}^{B_0;t_0} =
  D_t^{B_0;t_0}J_k +
  \max\left(D_t^{B_0;t_0}[(B_t^{B_0;t_0})^{\bullet k}]_+,\,\, D_t^{B_0;t_0}[(-B_t^{B_0;t_0})^{\bullet k}]_+ \right).
\end{equation}
Note that neither product of matrices inside the max in \eqref{usual D} has any nonzero entry outside the $k\th$ column.
It turns out that $\m$-vectors satisfy the same recursion, but with different initial conditions.

\begin{prop}
  \label{usual M}
  The initial $M$-matrix $M_{t_0}^{B_0;t_0}$ is the identity matrix.
  Given an edge $t\dashname{k}t'$ in $\T_n$,
  \[
    M_{t'}^{B_0;t_0} =
    M_t^{B_0;t_0}J_k +
    \max\left(M_t^{B_0;t_0}[(B_t^{B_0;t_0})^{\bullet k}]_+,\,\, M_t^{B_0;t_0}[(-B_t^{B_0;t_0})^{\bullet k}]_+ \right).
  \]
\end{prop}

Finally, we present some conjectures and results on Property  R in the context of source-sink moves.
Suppose that in the exchange matrix $B_0$, all entries in row $k$ weakly agree in sign.
That is, either all entries in row $k$ are nonnegative (and equivalently all entries in column $k$ are nonpositive) or all entries in row $k$ are nonpositive (and equivalently all entries in column $k$ are nonnegative).
In this case, mutation of $B_0$ in direction $k$ is often called a \newword{source-sink move}, referring to the operation on quivers of reversing all arrows at a source or a sink.
We conjecture that Property R holds when mutation at $k$ is a source-sink move.
In this case, equation \eqref{D backwards} has a particularly simple form.

\begin{conj}
  \label{source-sink conj}
  Suppose $t_0\dashname{k}t_1$ is an edge in $\T_n$ and $B_1$ is $\mu_k(B_0)$.
  If all entries in row $k$ of $B_0$ weakly agree in sign, then
  \begin{equation}
    \label{D backwards source-sink}
    D_t^{B_1;t_1} =
    J_kD_t^{B_0;t_0} +
    \bigl[|B_0^{k\bullet}|D_t^{B_0;t_0}\bigr]_+\\
  \end{equation}
\end{conj}

We also make two other closely related conjectures.
Let $A$ be the \newword{Cartan companion} of $B_0$, defined by setting $A_{ii}=2$ for all $i$ and $A_{ij}=-|(B_0)_{ij}|$ for $i\neq j$.
Then $A$ is a (generalized) Cartan matrix and thus defines a root system and a root lattice in the usual way.
It also defines a (generalized) Weyl group $W$, generated by simple reflections $s_1,\ldots,s_n$ given by $s_k(\alpha_\ell)=\alpha_\ell-A_{k\ell}\alpha_k$, where the $\alpha_i$ are the simple roots.
If $\beta$ is in the root lattice, then write $[\beta:\alpha_i]$ for the coefficient of $\alpha_i$ in the simple root coordinates of $\beta$.
Then $[s_k(\beta):\alpha_i]=[\beta:\alpha_i]$ if $i\neq k$ and $[s_k(\beta):\alpha_k]=-[\beta:\alpha_k]+\sum_{\ell=1}^n|(B_0)_{k\ell}|[\beta:\alpha_\ell]$.
Following \cite[Section~2]{ga}, we define a piecewise linear modification $\sigma_k$ of $s_k$ by setting $[\sigma_k(\beta):\alpha_i]=[\beta:\alpha_i]$ if $i\neq k$ and $[\sigma_k(\beta):\alpha_k]=-[\beta:\alpha_k]+\sum_{\ell=1}^n|(B_0)_{k\ell}|\bigl[[\beta:\alpha_\ell]\bigr]_+$.
We think of $\sigma_k$ as a map on (certain) integer vectors by interpreting them as simple root coordinates of vectors in the root lattice.
We also think of $\sigma_k$ as a map on integer matrices by applying it to each \emph{column}.
\begin{conj}
  \label{source-sink sigma}
  Suppose $t_0\dashname{k}t_1$ is an edge in $\T_n$ and $B_1$ is $\mu_k(B_0)$.
  If all entries in row $k$ of $B_0$ weakly agree in sign, then $D_t^{B_1;t_1}=\sigma_kD_t^{B_0;t_0}$.
\end{conj}

To relate Conjecture~\ref{source-sink sigma} to Conjecture~\ref{source-sink conj}, we quote the following conjecture, which is a significant weakening of \cite[Conjecture~7.4]{ca4}.
We will say a matrix $D$ \newword{has signed columns} if every column of $D$ either has all nonnegative entries or all nonpositive entries.
Similarly, $D$ \newword{has signed rows} if every row of $D$ either has all nonnegative entries or all nonpositive entries.

\begin{conj}
  \label{signed columns}
  For all $t\in\T_n$, the matrix $D_t^{B_0;t_0}$ has signed columns.
\end{conj}

Conjecture~\ref{signed columns} is not the same as another weakening of \cite[Conjecture~7.4]{ca4}, namely ``sign-coherence of $\d$-vectors,'' which asserts that for all $t\in\T_n$, the matrix $D_t^{B_0;t_0}$ has signed \emph{rows}.

We prove the following easy proposition in Section~\ref{general sec}.
\begin{prop}
  \label{equiv conj}
  If Conjecture~\ref{signed columns} holds, then Conjectures~\ref{source-sink conj} and \ref{source-sink sigma} are equivalent.
\end{prop}

Theorems~\ref{D rk 2} and \ref{D finite} imply Conjecture~\ref{source-sink conj} in the rank-two and finite-type cases, and Theorem~\ref{D surface} implies it for certain surfaces.
In \cite{Dylan}, Dylan Rupel proved Conjectures~\ref{source-sink sigma} and~\ref{signed columns} (and thus Conjecture~\ref{source-sink conj}) for $B$ acyclic.
As further evidence in support of the conjectures in general, we prove the following theorem in Section~\ref{source-sink surf sec}.

\begin{theorem}
  \label{source-sink surf}
  Conjectures~\ref{source-sink conj} and~\ref{source-sink sigma} hold in cluster algebras arising from marked surfaces.
\end{theorem}

\section{Proofs of general results}
\label{general sec}
We begin with the proof of Proposition~\ref{usual M}, followed by the proof of Proposition~\ref{MD init}.
To make the proof of Proposition~\ref{usual M} completely clear, we point out two lemmas about highest powers in multivariate (Laurent) polynomials.
Both are completely obvious when looked at in the right way, but otherwise one might convince oneself to worry.
Given a Laurent polynomial $p$, we write $m_i(p)$ for the highest power of $x_i$ occurring in a term of $p$.

\begin{lemma}
  \label{highest poly}
  Given Laurent polynomials $f$ and $g$ in $x_1,\ldots,x_n$, we have $m_i(fg)=m_i(f)+m_i(g)$.
\end{lemma}
\begin{proof}
  Write $f=f_a x_i^a+f_{a+1}x_i^{a+1}+\cdots+f_kx_i^k$ and $g=g_b x_i^b+g_{b+1}x_i^{b+1}+\cdots+g_\ell x_i^\ell$ such that the $f_j$ and $g_j$ are polynomials in the variables besides $x_i$ and $f_k$ and $g_\ell$ are nonzero.
  Then the highest power of $x_i$ in $fg$ is $k+\ell$.
  (Otherwise $f_k$ and $g_\ell$ are zero divisors.)
\end{proof}
\begin{lemma}
  \label{highest rat}
  Suppose $p$ is a Laurent polynomial over $\complexes$ in $x_1,\ldots,x_n$ and $f$ and $g$ are polynomials in $\complexes[x_1,\ldots,x_n]$ such that $f/g=p$.
  Then $m_i(p)=m_i(f)-m_i(g)$.
\end{lemma}
\begin{proof}
  Since $p$ is a Laurent polynomial, we can factor $f$ as $a\cdot c$ and $g$ as $b\cdot c$ such that $b$ is a monomial.
  It is immediate that $m_i(p)=m_i(a)-m_i(b)$.
  Applying Lemma~\ref{highest poly}, we have $m_i(f)-m_i(g)=m_i(a)+m_i(c)-m_i(b)-m_i(c)=m_i(p)$.
\end{proof}

\begin{proof}[Proof of Proposition~\ref{usual M}]
  Throughout this proof, we omit superscripts $B_0;t_0$.
  The first assertion of the proposition is trivial.
  To establish the second assertion, we compute $M_{ij;t'}$, the highest power of $x_i$ occurring in $x_{j;t'}$, in terms of $M_t$.
  If $j\neq k$, then $x_{j;t'}=x_{j;t}$, so $M_{ij;t'}=M_{ij;t}$ as given in the proposition.
  If $j=k$, then the exchange relation \cite[(2.8)]{ca4}, with trivial coefficients, is
  \begin{equation}
    \label{exchange}
    x_{k;t'} =
    \left(x_{k;t}\right)^{-1}
    \biggl(\prod_\ell \left(x_{\ell,t}\right)^{[B_{\ell k;t}]_+}+\prod_\ell \left(x_{\ell,t}\right)^{[-B_{\ell k;t}]_+}\biggr).
  \end{equation}
  Write $U$ for the expression $\prod_\ell \left(x_{\ell,t}\right)^{[B_{\ell  k;t}]_+}+\prod_\ell \left(x_{\ell,t}\right)^{[-B_{\ell k;t}]_+}$.
  Each factor $x_{\ell;t}$ in $U$ has a subtraction-free expression: an expression as a ratio of two polynomials in $x_1,\ldots,x_n$ with nonnegative coefficients.
  Therefore each term in $U$ has a subtraction-free expression.
  Write the first term as $a/c$ and the second term as $b/d$, where $a$, $b$, $c$, and $d$ are polynomials with nonnegative coefficients.
  The sum $U$ is then  $\frac{ad}{cd}+\frac{bc}{cd}$.
  Since all of these expressions are subtraction-free, there is no cancellation, so $m_i(U)=m_i(\frac{ad}{cd}+\frac{bc}{cd})=\max\left(m_i(\frac{ad}{cd}),m_i(\frac{bc}{cd})\right)$, which equals $\max\left(m_i(\frac{a}{c}),m_i(\frac{b}{d})\right)$ which in turn equals
  \[
    \max\Biggl(m_i\biggl(\prod_\ell
    (x_{\ell,t})^{[B_{\ell k;t}]_+}\biggr),m_i\biggl(\prod_\ell
    (x_{\ell,t})^{[-B_{\ell k;t}]_+}\biggr)\Biggr).
  \]
  Returning now to expressions for the $x_{\ell;t}$ as Laurent polynomials, Lemma~\ref{highest poly} lets us conclude that $m_i(U)=\max(\sum_\ell M_{i\ell;t}[B_{\ell k;t}]_+,\sum_\ell M_{i\ell;t}[-B_{\ell k;t}]_+)$.

  Now, writing $x_{k;t}$ as a rational function $p/q$ with $m_i(p)-m_i(q)=m_i(x_{k;t})$ and writing $U$ as a rational function $r/s$ with $m_i(r)-m_i(s)=m_i(U)$, equation~\eqref{exchange} lets us write $x_{k;t'}$ as $\frac{qr}{ps}$, so Lemmas~\ref{highest poly} and~\ref{highest rat} imply that $M_{ik;t'}=m_i(q)-m_i(p)+m_i(r)-m_i(s)=-M_{ik;t}+\max(\sum_\ell M_{i\ell;t}[B_{\ell k;t}]_+,\sum_\ell M_{i\ell;t}[-B_{\ell k;t}]_+)$ as desired.
\end{proof}

\begin{proof}[Proof of Proposition~\ref{MD init}]
  The cluster at $t_1$ is obtained from $(x_1,\ldots,x_n)$ by removing $x_k$ and replacing it with a new cluster variable $x'_k$.
  The two are related by
  \begin{equation}
    \label{init mut}
    x_k =
    (x'_k)^{-1}
    \biggl(\prod_\ell x_\ell^{[b_{\ell k}]_+}+\prod_\ell x_\ell^{[-b_{\ell k}]_+}\biggr),
  \end{equation}
  where the $b_{\ell k}$ are entries of $B_0$.

  To show that the $k\th$ row of $D_t^{B_1;t_1}$ equals the $k\th$ row of $M_t^{B_0;t_0}$, we appeal to the Laurent Phenomenon to write the cluster variable $x_{j;t}^{B_0;t_0}$ in the form $N(x_1,\ldots,x_n)/\prod_ix_i^{D_{ij;t}^{B_0;t_0}}$ for some polynomial $N$ not divisible by any of the $x_i$.
  We write $N=N_0+N_1x_k+\cdots+N_px_k^p$, where the $N_q$ are polynomials not involving $x_k$, with $N_p\neq0$.
  Then \eqref{init mut} lets us write $x_{j;t}^{B_0;t_0}$ as
  \begin{equation}
    \label{xjt subs}
    \frac{N_0+N_1\frac{\left(\prod_\ell x_\ell^{[b_{\ell k}]_+}+\prod_\ell
      x_\ell^{[-b_{\ell k}]_+}\right)}{x'_k}+\cdots+N_p\frac{\left(\prod_\ell
      x_\ell^{[b_{\ell k}]_+}+\prod_\ell x_\ell^{[-b_{\ell
        k}]_+}\right)^p}{(x'_k)^p}}
      {x_1^{D_{1j;t}^{B_0;t_0}}\cdots\left(\frac{\left(\prod_\ell x_\ell^{[b_{\ell
        k}]_+}+\prod_\ell x_\ell^{[-b_{\ell
      k}]_+}\right)}{x'_k}\right)^{D_{kj;t}^{B_0;t_0}}\cdots
      x_n^{D_{nj;t}^{B_0;t_0}}}
  \end{equation}
  The numerator of \eqref{xjt subs} can be factored as $(x'_k)^{-p}$ times a polynomial not divisible by $x'_k$.
  The denominator can be factored as $(x'_k)^{-D_{kj;t}^{B_0;t_0}}$ times a polynomial not involving $x'_k$.
  We conclude that $D_{kj;t}^{B_1,t_1}$ is $-D_{kj;t}^{B_0;t_0}+p$.
  The latter equals $M_{kj;t}^{B_0;t_0}$.

  To show that $D_t^{B_1;t_1}$ agrees with $D_t^{B_0;t_0}$ outside of row $k$, we fix $i\neq k$ and consider a subtraction-free expression for $x_{j;t}^{B_0;t_0}$.
  The Laurent Phenomenon implies that this expression can be simplified to a Laurent polynomial.
  The simplification can, if one wishes, be done in two stages, by first factoring out all powers of $x_i$ from the rational expression and then canceling the other factors.
  After the first stage, we have written $x_{j;t}^{B_0;t_0}$ as $x_i^{-D_{ij;t}^{B_0;t_0}}\cdot\frac{f}{g}$ where $f$ and $g$ are subtraction-free polynomials not divisible by  $x_i$.
  Replacing $x_k$ in this expression by the right side of \eqref{init mut}, we find that no additional powers of $x_i$ can be extracted.
  (Since the right side of \eqref{init mut} is also subtraction-free, we obtain a new subtraction-free expression.
  In particular, there can be no cancellation, so a power of $x_i$ can be extracted if and only if it is a factor in every term of the numerator or a factor in every term of the denominator.
  But the right side of \eqref{init mut} is not divisible by any nonzero power of $x_i$.)
  We conclude that $D_{ij;t}^{B_1,t_1}=D_{ij;t}^{B_0;t_0}$.
\end{proof}

We next prove Theorem~\ref{DRM}.
Specifically, the theorem follows from the next three propositions, which more carefully specify the relations among the three properties.

\begin{proposition}
  \label{D implies R}
  For a fixed choice of $B_0$, $t_0$, $t$ and $k$, let $t_1$ be the vertex of $\T_n$ such that $t_0\dashname{k}t_1$ and write $B_1$ for $\mu_k(B_0)$.
  Suppose equation~\eqref{DD} holds at $B_0,t_0,t$ and also at $B_1,t_1,t$.
  Then equation~\eqref{D backwards} holds for the same $B_0,t_0,t,k$.
\end{proposition}

\begin{proof}
  We apply \eqref{DD} at $B_1,t_1,t$, then \eqref{usual D}, then \eqref{DD} at $B_0,t_0,t$.
  \begin{align*}
    D_t^{B_1;t_1}
      & =\left(D_{t_1}^{(-B_t)^T;t}\right)^T       \\
      & =\left(D_{t_0}^{(-B_t)^T;t}J_k + \max\left(D_{t_0}^{(-B_t)^T;t}[(B_0^T)^{\bullet k}]_+,\,\, D_{t_0}^{(-B_t)^T;t}[(-B_0^T)^{\bullet k}]_+ \right)\right)^T\\
      & =J_k\left(D_{t_0}^{(-B_t)^T;t}\right)^T+\! \max\left([B_0^{k\bullet}]_+\left(D_{t_0}^{(-B_t)^T;t}\right)^T,\,\, [-B_0^{k\bullet}]_+\left(D_{t_0}^{(-B_t)^T;t}\right)^T \right)\\
      & =J_kD_t^{B_0;t_0} + \max\left([B_0^{k\bullet}]_+D_t^{B_0;t_0},\,\, [-B_0^{k\bullet}]_+D_t^{B_0;t_0} \right)
  \end{align*}
  In the second line, we use the fact that $B_{t_0}^{(-B_t)^T;t}=-B_0^T$.
\end{proof}

\begin{proposition}
  \label{R implies D}
  Fix a (coefficient-free) cluster pattern $t\mapsto(B_t,(x_{1;t},\ldots,x_{n;t}))$ and vertices $t_0$ and $t$ of $\T_n$, connected by edges
  \[
    t_0\dashname{k_1=k}t_1\dashname{k_2}\cdots\dashname{k_m}t_m=t.
  \]
  Suppose that, for all $i=1,\ldots,m$, equation~\eqref{D backwards} holds for the edge $t_{i-1}\dashname{k_i}t_i$.
  Then $\bigl(D_t^{B_{t_0};t_0}\bigr)^T=D_{t_0}^{(-B_t)^T;t}$.
\end{proposition}

\begin{proof}
  We argue by induction on $m$.
  For $m=0$ (i.e.\ $t=t_0$), equation~\eqref{DD} says that the negative of the identity matrix is symmetric.
  Equation \eqref{D backwards} is symmetric in switching $t_0$ and $t_1$, because $B_0^{k\bullet}=-B_1^{k\bullet}$.
  Thus for $m>0$, we can use \eqref{D backwards} for the edge $t_1\dashname{k_1} t_0$ to write
  \begin{equation}
    \label{use D backwards}
    D_t^{B_0;t_0}= J_kD_t^{B_1;t_1}+\max\Bigl([B_1^{k\bullet}]_+D_t^{B_1;t_1},\,\, [-B_1^{k\bullet}]_+D_t^{B_1;t_1} \Bigr)
  \end{equation}
  By induction, we rewrite the right side of \eqref{use D backwards} as
  \begin{align*}
    & J_k\Bigl(D_{t_1}^{(-B_t)^T;t}\Bigr)^T+\max\left([B_1^{k\bullet}]_+\left(D_{t_1}^{(-B_t)^T;t}\right)^T,\,\, [-B_1^{k\bullet}]_+\left(D_{t_1}^{(-B_t)^T;t}\right)^T \right) \\
    & \qquad= \left(D_{t_1}^{(-B_t)^T;t}J_k+\max\left(D_{t_1}^{(-B_t)^T;t}[(B_1^T)^{\bullet k}]_+,\,\, D_{t_1}^{(-B_t)^T;t}[(-B_1^T)^{\bullet k}]_+\right)\right)^T.
  \end{align*}
  By \eqref{usual D}, this is $\bigl(D_{t_0}^{(-B_t)^T;t}\bigr)^T$.
\end{proof}

\begin{proposition}
  \label{R iff M}
  For a fixed choice of $B_0$, $t_0$, $t$ and $k$, equation~\eqref{D backwards} holds if and only if equation~\eqref{MD} holds \emph{in the $k\th$ row}.
\end{proposition}
\begin{proof}
  Equation~\eqref{MD} holds in the $k\th$ row if and only if
  \begin{equation*}
    \label{MD k}
    (M_t^{B_0;t_0})^{k\bullet}=(-D_t^{B_0;t_0})^{k\bullet}+\max\left([B_0^{k\bullet}]_+D_t^{B_0;t_0},\,\, [-B_0^{k\bullet}]_+D_t^{B_0;t_0} \right).\\
  \end{equation*}
  This equation is equivalent to \eqref{D backwards} in light of Proposition~\ref{MD init}.
\end{proof}

This completes the proof of Theorem~\ref{DRM}.

To conclude this section, we establish Proposition~\ref{equiv conj} by proving a more detailed statement.
Recall that a matrix $D$ has signed columns if every column of $D$ either has all nonnegative entries or all nonpositive entries.
\begin{prop}
  \label{equiv conj detailed}
  Suppose $t_0\dashname{k}t_1$ is an edge in $\T_n$ and $B_1$ is $\mu_k(B_0)$.
  If $D_t^{B_0;t_0}$ has signed columns, then the right side of \eqref{D backwards source-sink} equals $\sigma_kD_t^{B_0;t_0}$.
\end{prop}
\begin{proof}
  Let $\beta$ be the vector in the root lattice with simple root coordinates $\d_{j;t}^{B_0;t_0}$.
  For $i\neq k$, the $ij$-entry of the right side of \eqref{D backwards source-sink} is $[\beta:\alpha_i]$.
  The $kj$-entry of the right side of \eqref{D backwards source-sink} is $-[\beta:\alpha_k]+\Bigl[\sum_{\ell=1}^n|(B_0)_{k\ell}|[\beta:\alpha_\ell]\Bigr]_+.$
  By hypothesis, all of the simple root coordinates of $\beta$ weakly agree in sign, so $\bigl[\sum_{\ell=1}^n|(B_0)_{k\ell}|[\beta:\alpha_\ell]\bigr]_+$ is $\sum_{\ell=1}^n|(B_0)_{k\ell}|\bigl[[\beta:\alpha_\ell]\bigr]_+$.
  Thus the right side of \eqref{D backwards source-sink} is $\sigma_k\beta$.
\end{proof}

\section{Duality and recursion in certain cluster algebras}
\label{special sec}
We now prove Theorems~\ref{D rk 2}, \ref{D finite}, and~\ref{D surface}.

\subsection{Rank two}
The proof of Theorem~\ref{D rk 2} uses a formula for rank-two denominator vectors due to Lee, Li, and Zelevinsky \cite[(1.13)]{LLZ}.

\begin{proof}[Proof of Theorem~\ref{D rk 2}]
  The $2$-regular tree $\T_n$ is an infinite path.
  We label its vertices $t_k$ for $k\in\integers$, and abbreviate $B_{t_k}$ by $B_k$.
  As the situation is very symmetric, it is enough to take $B_0=\begin{bsmallmatrix}0&b\\-c&0\end{bsmallmatrix}$ with $b$ and $c$ nonnegative and establish \eqref{DD} for $t=t_k$ with $k\ge0$.
  When $bc<4$, the cluster pattern is of finite type and \eqref{DD} can be checked easily (if a bit tediously) by hand.
  Alternatively, one can appeal to Theorem~\ref{D finite}, which we prove below.
  For $bc\ge4$, the denominator vectors are given by \cite[(1.13)]{LLZ}.
  Equation~\eqref{DD} is easy when $k=1$, so we assume $k\ge 2$.
  The labeled cluster associated to the vertex $t_k$ is $\{x_{k+1},x_{k+2}\}$ if $k$ is even and $\{x_{k+2},x_{k+1}\}$ if $k$ is odd.

  If $k$ is even, we use \cite[(1.13)]{LLZ} to write
  \begin{equation}
    \label{Dt even}
    D_{t_k}^{B_0;t_0}=\begin{bmatrix}
      S_{\frac{k-2}{2}}(u)+S_{\frac{k-4}{2}}(u)   &bS_{\frac{k-2}{2}}(u)\\
      cS_{\frac{k-4}{2}}(u)    &S_{\frac{k-2}{2}}(u)+S_{\frac{k-4}{2}}(u)
    \end{bmatrix}
  \end{equation}
  where $u=bc-2$ and the $S_p$ are Chebyshev polynomials of the second kind.
  (In fact, here we do not need to know anything about the $S_p$ except that they are functions of $u$.)
  We can similarly use \cite[(1.13)]{LLZ} to write an expression for $D_{t_0}^{-B_k^T;t_k}$.
  Since $k$ is even $B_k=B_0$, and thus $-B_k^T=-B_0^T=\begin{bsmallmatrix}0&c\\-b&0\end{bsmallmatrix}$.
  To apply \cite[(1.13)]{LLZ} in this case, we must switch the role of $b$ and $c$.
  When we do so, keeping in mind that we move now in the negative direction, we obtain exactly the transpose of the right side of \eqref{Dt even}.

  If $k$ is odd, we obtain
  \begin{equation}
    \label{Dt odd}
    D_{t_k}^{B_0;t_0}=\begin{bmatrix}
      S_{\frac{k-1}{2}}(u)+S_{\frac{k-3}{2}}(u)   &bS_{\frac{k-3}{2}}(u)\\
      cS_{\frac{k-3}{2}}(u)    &S_{\frac{k-3}{2}}(u)+S_{\frac{k-5}{2}}(u)
    \end{bmatrix}
  \end{equation}
  In this case, $B_k=-B_0$, so $-B_k^T=B_0^T=\begin{bsmallmatrix}0&-c\\b&0\end{bsmallmatrix}$.
  Noticing that $-B_k^T$ is obtained from $B_0$ by simultaneously swapping the rows and the columns, when we use \cite[(1.13)]{LLZ} to write an expression for $D_{t_0}^{-B_k^T;t_k}$, we also swap the rows and columns.
  The result is exactly the transpose of the right side of \eqref{Dt odd}.
\end{proof}

\subsection{Finite type}
The proof of Theorem~\ref{D finite} uses a result of Ceballos and Pilaud~\cite{CP} giving denominator vectors in finite type, with respect to any initial seed, in terms of the compatibility degrees defined at any acyclic seed.
In \cite{ca2}, it is shown that in every cluster pattern of finite type, there exists an exchange matrix $B_0$ that is bipartite and whose Cartan companion $A$ is of finite type.
The cluster variables appearing in the cluster pattern are in bijection with the almost positive roots in the root system for $A$.
Given an almost positive root $\beta$, we will write $x(\beta)$ for the corresponding cluster variable.
There is a \newword{compatibility degree} $(\alpha,\beta)\mapsto(\alpha\parallel\beta)\in\integers_{\ge0}$ defined on almost positive roots encoding some of the combinatorial properties of the cluster algebra. In particular two cluster variables $x(\alpha)$ and $x(\beta)$ belong to the same cluster if and only if the roots $\alpha$ and $\beta$ are \newword{compatible} (i.e. if their compatibility degree is zero).
Maximal sets of compatible roots are called \newword{(combinatorial) clusters} and they correspond to the \newword{(algebraic) clusters} in the cluster algebra.
In the same paper Fomin and Zelevinsky also showed that compatibility degrees encode denominator vectors with respect to the bipartite initial seed.

Ceballos and Pilaud extended this result dramatically in the following result, which is \cite[Corollary~3.2]{CP}.
(We follow Ceballos and Pilaud in modifying the definition of compatibility degree in an inconsequential way in order to make it easier to state the theorem.
Specifically, we take $(\alpha\parallel\alpha)=-1$ rather than $(\alpha\parallel\alpha)=0$.)

\begin{theorem}
  \label{CP cor}
  Let $\set{\beta_1,\ldots,\beta_n}$ be a cluster and let $\gamma$ be an almost positive root.
  Then the $\d$-vector of $x(\gamma)$ with respect to the cluster $\set{x(\beta_1),\ldots,x(\beta_n)}$ is given by $[(\beta_1\parallel\gamma),\ldots,(\beta_n\parallel\gamma)]$.
\end{theorem}

Since $B_0$ is skew-symmetrizable, passing from $B_0$ to $-B_0^T$ has the effect of preserving the signs of entries while transposing the Cartan companion  $A$.
The almost positive roots for $A^T$ are the almost positive co-roots associated to $A$.
The following is \cite[Proposition~3.3(1)]{ga}.

\begin{prop}
  \label{compat ck}
  If $\alpha$ and $\beta$ are almost positive roots and $\alpha\ck$ and $\beta\ck$ are the corresponding co-roots, then $(\alpha\parallel\beta)=(\beta\ck\parallel\alpha\ck)$.
\end{prop}

\begin{proof}[Proof of Theorem~\ref{D finite}]
  The cluster pattern assigns some algebraic cluster to $t_0$ and some algebraic cluster to $t$, and each of the algebraic clusters is encoded by some combinatorial cluster.
  Let $\set{\beta_1,\ldots,\beta_n}$ be the combinatorial cluster at $t_0$ and let $\set{\gamma_1,\ldots,\gamma_n}$ be the combinatorial cluster at $t$.
  Now Theorem~\ref{CP cor} and Proposition~\ref{compat ck} are exactly Property D at $t_0$ and $t$.
\end{proof}

\subsection{Marked surfaces}
\label{D surface sec}
The proof of Theorem~\ref{D surface} relies on a result of Fomin, Shapiro, and Thurston \cite[Theorem~8.6]{FST} giving denominator vectors in terms of tagged arcs.
We will assume familiarity with the basic definitions of cluster algebras arising from marked surfaces.

Recall that tagged arcs are in bijection with cluster variables and tagged triangulations are in bijection with clusters, except in the case of once-punctured surfaces with no boundary components, where plain-tagged arcs are in bijection with cluster variables and plain-tagged triangulations are in bijection with clusters.
We write $\alpha\mapsto x(\alpha)$ for this bijection.
Given tagged arcs $\alpha$ and $\beta$, there is an \newword{intersection number} $(\alpha|\beta)$ such that the following theorem \cite[Theorem~8.6]{FST} holds.
\begin{theorem}
  \label{FST denom}
  Given tagged arcs $\alpha$ and $\beta$ and a cluster $(x_1,\ldots,x_n)$ with $x_i=x(\alpha)$, the $i\th$ component of the denominator vector of $x(\beta)$ with respect to the cluster $(x_1,\ldots,x_n)$ is $(\alpha|\beta)$.
\end{theorem}

In an exchange pattern arising from a marked surface, every exchange matrix $B_t$ is skew-symmetric, so $(-B_t)^T=B_t$.
Thus we have the following corollary to Theorem~\ref{FST denom}.

\begin{cor}
  \label{D sym}
  In an exchange pattern arising from a marked surface, Property D holds if and only if the intersection number is symmetric (i.e.\
  $(\alpha|\beta)=(\beta|\alpha)$ on all tagged arcs $\alpha$ and $\beta$ that correspond to cluster variables).
\end{cor}

The intersection number $(\alpha|\beta)$ is defined in \cite[Definition~8.4]{FST} to be the sum of four quantities $A$, $B$, $C$, and $D$.
To define these, we choose $\alpha_0$ and $\beta_0$ to be non-self-intersecting curves homotopic (relative to the set of marked points) to $\alpha$ and $\beta$, and intersecting with each other the minimum possible number of times, transversally each time.
The quantity $A$ is the number of intersection points of $\alpha_0$ and $\beta_0$ (excluding intersections at their endpoints).
The quantity $B$ is zero unless $\alpha_0$ is a loop (i.e.\ unless the two endpoints of $\alpha_0$ coincide).
If $\alpha_0$ is a loop, let $a$ be its endpoint.
We number the intersections as $b_1,\ldots,b_k$ in the order they are encountered when following $\beta_0$ in some direction.
For each $i=1,\ldots,k-1$, there is a unique segment $[a,b_i]$ of $\alpha_0$ having endpoints $a$ and $b_i$ and not containing $b_{i+1}$.
There is also a unique segment $[a,b_{i+1}]$ of $\alpha_0$ having endpoints $a$ and $b_{i+1}$ and not containing $b_i$.
Let $[b_i,b_{i+1}]$ be the segment of $\beta_0$ connecting $b_i$ to $b_{i+1}$.
The quantity $B$ is $\sum_{i=1}^{k-1}B_i$, where $B_i$ is $-1$ if the segments $[a,b_i]$, $[a,b_{i+1}]$, and $[b_i,b_{i+1}]$ define a triangle that is contractible and $B_i=0$ otherwise.
The quantity $C$ is zero unless $\alpha_0$ and $\beta_0$ are equal up to isotopy relative to the set of marked points, in which case $C=-1$.
The quantity $D$ is the number of ends of $\beta$ that are incident to an endpoint of $\alpha$ and carry, at that endpoint, a different tag from the tag of $\alpha$ at that endpoint.

The quantities $A$ and $C$ are patently symmetric in $\alpha$ and $\beta$, so we need not consider them in this section.
It is pointed out in \cite[Example~8.5]{FST} that $D$ can fail to be symmetric.
The quantity $B$ can also fail to be symmetric.
Examples will occur below.

Some immediate observations will be helpful.
\begin{obs}
  \label{no loops}
  In a surface having no tagged arcs that are loops, $B$ is always $0$ and $D$ is also always symmetric.
\end{obs}

\begin{obs}
  \label{no punctures}
  In a surface having no punctures, $D$ is always zero.
\end{obs}

\begin{obs}
  \label{once-punctured}
  In a surface having exactly one puncture and no boundary components, $D$ is always zero on tagged arcs corresponding to cluster variables.
\end{obs}

For the second observation, recall that notched tagging may occur only at punctures.
For the third observation, recall that in a surface with exactly one puncture and no boundary components, tagged arcs correspond to cluster variables if and only if they are tagged plain.

To prove one direction of Theorem~\ref{D surface}, we show that $B+D$ is symmetric in the cases listed in the theorem.
First, recall that a tagged arc may not bound a once-punctured monogon and may not be homotopic to a segment of the boundary between two adjacent marked points.
In particular, there are no loops in a disc with at most one puncture, in the unpunctured annulus with 2 marked points, in the twice-punctured disk with one marked point on its boundary, or in the four-times-punctured sphere.
Thus Observation~\ref{no loops} shows that $B+D$ is symmetric in the cases described in~\eqref{disk good} and~\eqref{sphere 4 good}, and in the simplest cases described in~\eqref{small Atilde good} and~\eqref{small Dtilde good}.

In the remaining cases described in~\eqref{small Atilde good}, Observation~\ref{no punctures} shows that $D$ is always zero.
We are interested in pairs of arcs containing at least one loop (otherwise $B$ is zero in both directions).
Because an arc may not be homotopic to a boundary segment, there are no loops at a marked point if it is the only marked point on its boundary component.
A marked point that is not the only marked point on its boundary component supports exactly one loop.
If there are two marked points on one component and one marked point on the other, then we are in the situation of Figure~\ref{two and one}.
In this case, numbering the points as in the figure, the only two loops in the surface are based one at $1$ and one at $2$.
The remaining arcs start at $3$, spiral around some number of times and then reach either $1$ or $2$.
The only arc that intersects one of the loops more than once is the other loop.
These have $B=-1$ in both directions, so $B$ is symmetric in this case.
If there are two marked points on each boundary component, the argument is similar and only slightly more complicated.
There are four loops, as illustrated in Figure~\ref{two and two}.
Each of the remaining arcs connects a point of one boundary to a point of the other boundary, with some number of spirals.
Again, for any of the four loops there is only one other arc intersecting it more than once; it is the loop based at the other marked point on the same boundary component.
For each pair of intersecting loops we calculate $B=-1$ in both directions.
We have finished case \eqref{small Atilde good}.
\begin{figure}
  \begin{minipage}{110 pt}
    \centering
    \scalebox{0.9}{\includegraphics{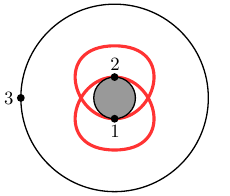}}
    \caption{}
    \label{two and one}
  \end{minipage}
  \begin{minipage}{120 pt}
    \centering
    \scalebox{0.9}{\includegraphics{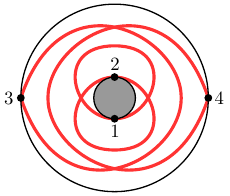}}
    \caption{}
    \label{two and two}
  \end{minipage}
  \begin{minipage}{110 pt}
    \centering
    \scalebox{0.9}{\includegraphics{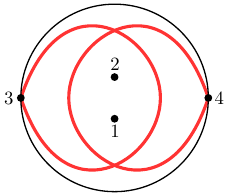}}
    \caption{}
    \label{two and punctures}
  \end{minipage}
  \label{annuli}
\end{figure}

The remaining case (a disk with two punctures and two boundary points) in \eqref{small Dtilde good} is similar to the cases in \eqref{small Atilde good}.
There is a loop at each marked point on the boundary but no other loop, as illustrated in Figure~\ref{two and punctures}.
In particular, the tagging at these loops is plain, and we see that $D$ is symmetric.
There is exactly one arc connecting the two boundary points and four tagged arcs (all with the same underlying arc) connecting the two punctures.
The remaining arcs have a boundary point at one endpoint and spiral around the punctures some number of times before ending at one of the punctures, with either tagging there.
Once again, the only arc that intersects more than once one of the loops is the other loop, and we again have $B=-1$ in both directions.

The remaining two cases are described in \eqref{torus 1 good}.
We first consider the once-punctured torus.
In this cases, $D=0$ by Observation~\ref{once-punctured}, so it remains to show that $B$ is symmetric.
We will show that in fact $B$ is zero on all pairs of arcs.
Arcs in the once-punctured torus are well-known to be in bijection with rational slopes, including the infinite slope.
(See, for example, \cite[Section~4]{unitorus}.)
Each such slope can be written uniquely as a reduced fraction $\frac ba$ such that $a\ge0$ and that $b=1$ whenever $a=0$.
If we take the universal cover (the plane $\reals^2$) of the torus mapping each integer point to the puncture, the arc indexed by a slope $\frac ba$ lifts to a straight line segment connecting the origin to the point $(a,b)$.
(The same arc also lifts to all integer translates of that line segment.)

It is now easy to see that $B=0$ for arcs in the once-punctured torus.
For any two arcs $\alpha$ and $\beta$, let $\alpha_0$ and $\beta_0$ be the curves on the torus obtained by projecting the associated straight line segments in the plane.
This choice of representatives minimizes the number of intersections as can be seen by looking at the universal cover.
Let $a$ and $b_1,\ldots,b_k$ be the points as in the definition of $B$.
Given some $i$ between $1$ and $k-1$, concatenate the curves $[a,b_i]$, $[b_i,b_{i+1}]$, and $[b_{i+1},a]$, and consider the lift of the concatenated curve to the plane.
This lifted curve consists of two parallel line segments and one line segment not parallel to the other two.
In particular, it is impossible for the lifted curve to start and end at the same point.
The situation is illustrated in Figure~\ref{one pun}, where a lift of $\beta_0$ is shown as a solid line, several lifts of $\alpha_0$ are shown as dotted lines and a lift of the three concatenated curves is highlighted.
By the standard argument on fundamental groups and universal covers, we see that the concatenation of $[a,b_i]$, $[b_i,b_{i+1}]$, and $[b_{i+1},a]$ is not a contractible triangle, and we conclude that $B=0$ on $\alpha$ and $\beta$.
\begin{figure}
  \begin{minipage}{164pt}
    \centering
    \scalebox{1.4}{\includegraphics{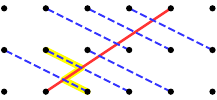}}
    \caption{}
    \label{one pun}
  \end{minipage}
  \begin{minipage}{167pt}
    \scalebox{0.4}{\includegraphics{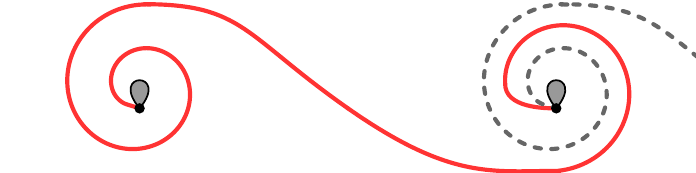}}\\[10pt]
    \scalebox{0.4}{\includegraphics{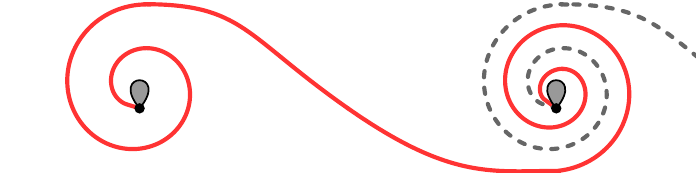}}
    \caption{}
    \label{two possible}
  \end{minipage}
\end{figure}

The final case for this direction of the proof is the torus with one boundary component and one marked point.
In this case, $D$ is again zero, this time by Observation~\ref{no punctures}, so we will show that $B$ is symmetric.
We think of the boundary component as a ``fat point'' on the torus.
With this trick, we can again consider lifts of arcs to the plane.
Each arc lifts to a curve connecting the origin to an integer point $(a,b)$ with $a$ and $b$ satisfying the same conditions as above for the once-punctured torus.
However, for each such $(a,b)$, there is a countable collection of arcs connecting the origin to $(a,b)$.
Specifically, for each integer $k$, the arc may wind $k$ times clockwise about the fat origin point before going to $(a,b)$.
(Negative values of $k$ specify counterclockwise spirals.)
Since $(a,b)$ and the origin both project to the same fat point on the torus, the number and direction of spirals at $(a,b)$ is determined almost uniquely by $k$.
There are two possibilities for each $k$, illustrated in Figure~\ref{two  possible} for the case where $(a,b)=(1,0)$.

For each arc $\alpha$, choosing the right change of basis of the integer lattice, we may as well assume that the lift of $\alpha$ connects the origin to the point $(1,0)$.
Furthermore, there is a homeomorphism from the torus to itself that rotates the fat point and changes the number of spirals of $\alpha$ at the origin and at $(1,0)$.
Rotating a half-integer number of full turns, we can assume $\alpha$ lifts to a straight horizontal line segment from $(0,0)$ to $(1,0)$.
Possibly reflecting the plane through the horizontal line containing the origin (to offset the effect of a half-turn), we can assume that $\alpha$ looks like the solid arc shown in Figure~\ref{spirals}, with the boundary component above the origin in the picture.
\begin{figure}
  \scalebox{0.7}{\includegraphics{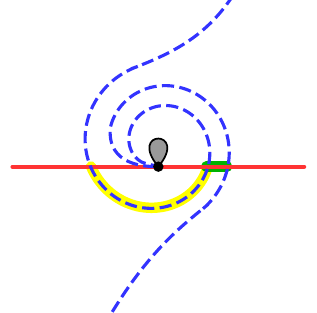}}\qquad\scalebox{0.7}{\includegraphics{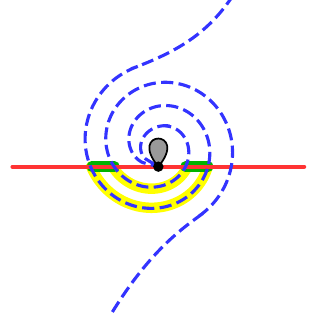}}
  \caption{}
  \label{spirals}
\end{figure}

Now take another arc $\beta$ and consider a lift of $\beta$ connecting the origin to $(a,b)$.
Since another lift connects $(-a,-b)$ to the origin, we may as well take $b\ge0$.
Up to a reflection in a vertical line, we can assume that the lift of $\beta$ spirals clockwise (if it spirals at all) as it leaves the origin.
Fixing one possible number of spirals of $\beta$ at $(0,0)$ and fixing some $(a,b)$ with $b>0$, the two possibilities for the lift of $\beta$ are shown as dashed arcs in Figure~\ref{spirals}.
Nonzero contributions to $B$ can arise only from segments that remain close to the fat point: by the same argument as for the once-punctured torus, the segments that do not stay near the fat point contribute nothing.
Therefore it is enough to analyze the intersections of $\alpha$ and $\beta$ near the origin.
In each of the two possibilities we highlight in Figure~\ref{spirals} the segments  $[a_i,a_{i+1}]$ of $\alpha$ and  $[b_j,b_{j+1}]$ of $\beta$ giving nonzero contributions.
In the pictured examples, $B$ is symmetric in $\alpha$ and $\beta$.
It is easy to see that the symmetry survives when the number of spirals changes.
The case $b=0$ looks slightly different, but $B$ is still symmetric for essentially the same reasons.
(Look back, for example at Figure~\ref{two possible}.)

We have proved one direction of Theorem~\ref{D surface}.
To prove the other direction, we need to show that $B+D$ fails to be symmetric in certain cases.
In each case, the failure of symmetry can be illustrated in a figure.
Here, we list the cases and indicate, for each case, the corresponding figure.
In some cases, we also include some comments in italics.
In each case, $\alpha$ is the solid arc and $\beta$ is the dashed arc; they intersect in at most two points. We omit the labeling $a, a_1, a_2, b, b_1, b_2$ not to clutter the pictures.
This will complete the proof of Theorem~\ref{D surface}.
\begin{enumerate}
  \item A surface with genus greater than 1 (Figure~\ref{genus}).
    \begin{figure}
      \begin{minipage}{106pt}
        \includegraphics{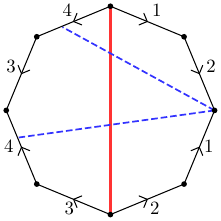}
        \caption{}
        \label{genus}
      \end{minipage}
      \qquad
      \qquad
      \begin{minipage}{191pt}
        \includegraphics{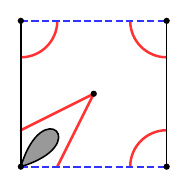}\quad\includegraphics{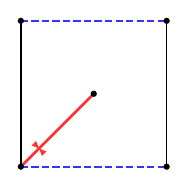}
        \caption{}
        \label{torus2}
      \end{minipage}
    \end{figure}
    \textit{
      We show the genus-2 case.
      Pairs of edges in the octahedron are identified as indicated by the numbering and the arrows.
      Since all taggings are plain, $D=0$.
      However, $B$ is asymmetric ($-1$ in one direction and $0$ in the other).
      The marked point shown in the figure is a puncture, but the same example works with the marked point on a boundary component.
      For higher genus or to have additional punctures, one can start with the surface shown and perform a connected sum, cutting a disk from the interior of the octagon shown.
    }
  \item A torus with 2 or more marked points (Figure~\ref{torus2}).
    \textit{
      Opposite pairs of edges in the square are identified.
      If the marked point at the corners of the square is on a boundary component, then the arcs shown in the left picture of the figure have $D=0$ but $B$ is asymmetric (taking values $0$ and $-1$).
      Additional punctures and/or boundary components may exist, but the arcs $\alpha$ and $\beta$ can always be chosen so that the triangle $[b,a_1]$,$[a_1,a_2]$, $[a_2,b]$ is contractible while the triangle $[a,b_1]$, $[b_1,b_2]$, $[b_2,a]$ is not.
      If the marked point at the corners is a puncture, then the right picture applies.
      In this case, $B=0$ but $D$ is asymmetric because one of the arcs is a loop and the other is not.
    }
  \item A sphere with 3 or more boundary components and possibly some punctures
    (Figure~\ref{sphere3}).
    \begin{figure}
      \begin{minipage}{100pt}
        \scalebox{0.9}{\includegraphics{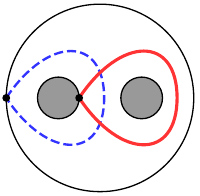}}
        \caption{}
        \label{sphere3}
      \end{minipage}
      \qquad
      \begin{minipage}{100pt}
        \scalebox{0.9}{\includegraphics{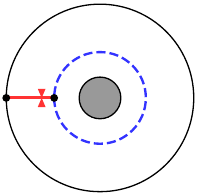}}
        \caption{}
        \label{annulus1}
      \end{minipage}
      \qquad
      \begin{minipage}{100pt}
        \scalebox{0.9}{\includegraphics{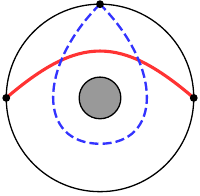}}
        \caption{}
        \label{annulus0}
      \end{minipage}
    \end{figure}
    \textit{
      We show a disk with 2 additional boundary components.
      For the arcs shown, $D=0$ but $B$ is asymmetric.
      Again, additional punctures and/or boundary components may exist, but the triangle $[a,b_1]$, $[b_1,b_2]$, $[b_2,a]$ is contractible.
    }
  \item An annulus with one or more punctures (Figure~\ref{annulus1}).
    \textit{
      $B=0$ and $D$ is asymmetric on the arcs shown.
    }
  \item An unpunctured annulus with 3 or more marked points on one of its boundary components (Figure~\ref{annulus0}).
    \textit{
      $B$ is asymmetric and $D=0$.
    }
  \item A disk with 3 or more punctures (Figure~\ref{disk3}).
    \begin{figure}
      \begin{minipage}{100pt}
        \scalebox{0.9}{\includegraphics{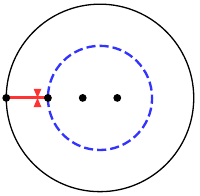}}
        \caption{}
        \label{disk3}
      \end{minipage}
      \qquad
      \begin{minipage}{100pt}
        \scalebox{0.9}{\includegraphics{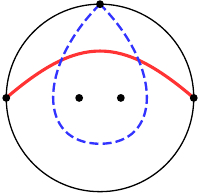}}
        \caption{}
        \label{disk2}
      \end{minipage}
    \end{figure}
    \textit{
      $B=0$ and $D$ is asymmetric.
    }
  \item A disk with 2 punctures and 3 or more marked points on the boundary (Figure~\ref{disk2}).
    \textit{
      $B$ is asymmetric and $D=0$.
    }
  \item A sphere with 5 or more punctures (Figure~\ref{sphere5}).
    \begin{figure}
      \includegraphics{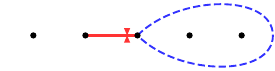}
      \caption{}
      \label{sphere5}
    \end{figure}
    \textit{
      We show a local patch of the sphere containing all of the punctures.
      $B=0$ and $D$ is asymmetric.
    }
\end{enumerate}

\section{Source-sink moves on triangulated surfaces}
\label{source-sink surf sec}
In this section, we prove Theorem~\ref{source-sink surf}, the assertion that Conjectures~\ref{source-sink conj} and~\ref{source-sink sigma} hold for marked surfaces.
Conjecture~\ref{signed columns} holds for surfaces because the stronger conjecture \cite[Conjecture~7.4]{ca4} for surfaces is an easy consequence of \cite[Theorem~8.6]{FST}.
Thus by Proposition~\ref{equiv conj}, we need only to prove the assertion about Conjecture~\ref{source-sink conj}.
In light of Theorem~\ref{FST denom}, the task is to prove a certain identity on intersection numbers.
This identity is already known (as a special case of Property R) for the surfaces listed in Theorem~\ref{D surface}, and it will be convenient in what follows that we need not consider those surfaces.

Suppose $\alpha$ is a tagged arc in a tagged triangulation $T$ and suppose $\alpha'$ is the arc obtained by flipping $\alpha$ in $T$.
We may as well take $T$ to be obtained from an ideal triangulation $T^\circ$ by applying the map $\tau$ of \cite[Definition~7.2]{FST} to each arc.
(Any other tagged triangulation could be obtained from such a triangulation by changing tags, which by definition \cite[Definition~9.6]{FST} does not affect the associated $B$-matrix.)
In particular, $B(T)=B(T^\circ)$.
We will abuse notation and denote by the same Greek letters both ideal arcs and their corresponding tagged arcs.
Suppose all of the entries in the row of $B(T)$ indexed by $\alpha$ weakly agree in sign.
Because of the symmetry between $D_t^{B_0;t_0}$ and $D_t^{B_1;t_1}$ in \eqref{D backwards source-sink}, we may as well assume that all entries in the row of $B(T)$ indexed by $\alpha$ are nonnegative; in this case we will say that ``$\alpha$ is a source'' alluding to the usual encoding of skew-symmetric exchange matrices by quivers.

Let $\beta$ be any other arc.
Keeping in mind that equation~\eqref{D backwards source-sink}, like equation~\eqref{D backwards} before it, is true outside of row $k$ by Proposition~\ref{MD init}, the task is to prove the following identity:
\begin{equation}
  \label{the task}
  (\alpha'|\beta)=-(\alpha|\beta)+\sum_{\gamma\in T}b_{\alpha\gamma}(\gamma|\beta)
\end{equation}
where $b_{\alpha\gamma}$ is the entry of $B(T)$ in the row indexed by $\alpha$ and column indexed by $\gamma$.

The key observation in our proof is that the entries $b_{\alpha\gamma}$ depends only on how $T$ looks locally near $\alpha$.
Therefore we begin our analysis by constructing a short list of possible local configurations.
To do this we build the surface and the ideal triangulation $T$ simultaneously by adjoining \newword{puzzle pieces} as in \cite[Section~4]{FST}.
In fact, in \cite[Section~4]{FST}, the \emph{ideal} triangulation $T^\circ$ is built from puzzle pieces, but to save a step, we apply the map $\tau$ to the puzzle pieces \emph{before} assembling, rather than \emph{after}.
The resulting \newword{tagged puzzle pieces} are shown in Figure~\ref{tagged pieces}.
We will refer to them (from left to right in the Figure) as triangle pieces, digon pieces, and monogon pieces.
The external edges of digon pieces are distinguishable (up to reversing the orientation of the surface) and we will call them the left edge and the right edge according to how they are pictured in Figure~\ref{tagged pieces}.
Similarly, the two pairs of internal arcs in a monogon piece are distinguishable, and we will call them the left pair and right pair according to Figure~\ref{tagged pieces}.
\begin{figure}
  \scalebox{0.9}{\includegraphics{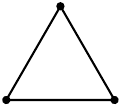}}
  \qquad
  \qquad
  \scalebox{0.9}{\includegraphics{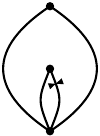}}
  \qquad
  \qquad
  \scalebox{0.9}{\includegraphics{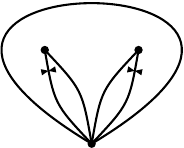}}
  \caption{Tagged puzzle pieces}
  \label{tagged pieces}
\end{figure}
Puzzle pieces are joined by gluing along their outer edges.
Unjoined outer edges become part of the boundary of the surface.
In \cite[Section~4]{FST}, one specific triangulation is mentioned that cannot be obtained from these puzzle pieces, but it is a triangulation of the 4-times punctured sphere, so by Theorem~\ref{D surface}, we need not consider it.

The list of possible local configurations around $\alpha$, given $\alpha$ is a source, appears in Figure~\ref{configurations}. (We leave  out the cases where Theorem~\ref{D surface} applies.)
In the figure, areas just outside the boundary are marked in gray.
The curve $\alpha$ is labeled, or if two curves might be a source, both of them are labeled $\alpha$.

\begin{figure}
  \scalebox{0.9}{\includegraphics{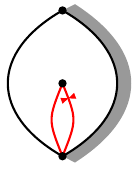}%
  \begin{picture}
    (0,0)(35,-43)
    \put(7,-18){\small\textcolor{red}{$\alpha$}}\put(-13,-18){\small\textcolor{red}{$\alpha$}}
  \end{picture}}
  \scalebox{0.9}{\includegraphics{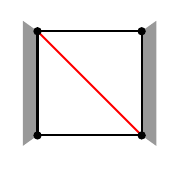}%
  \begin{picture}
    (0,0)(43,-43)
    \put(1,1){\small\textcolor{red}{$\alpha$}}
  \end{picture}}
  \,\,
  \scalebox{0.9}{\includegraphics{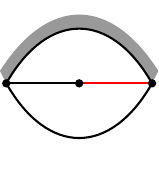}%
  \begin{picture}
    (0,0)(38,-43)
    \put(12,2.5){\small\textcolor{red}{$\alpha$}}
  \end{picture}}
  \,\,\,
  \scalebox{0.9}{\includegraphics{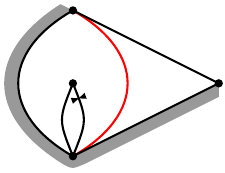}%
  \begin{picture}
    (0,0)(68,-43)
    \put(23,-1){\small\textcolor{red}{$\alpha$}}
  \end{picture}}
  \caption{Possible local configurations surrounding a source}
  \label{configurations}
\end{figure}
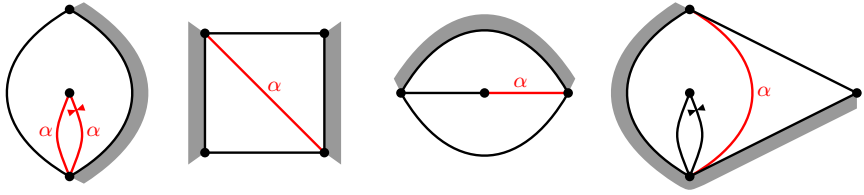

To obtain this list, recall that the entries in the row indexed by $\alpha$ are determined by the triangles of $T^\circ$ containing $\alpha$ or, if $\alpha$ is the folded side of a self-folded triangle, by the triangles containing the other side of that self-folded triangle.
(See \cite[Definition~4.1]{FST}.)
In particular, if $\alpha$ is an internal arc in a digon or monogon piece, the entries in the row indexed by $\alpha$ are determined completely within the piece.
Both internal arcs in the digon piece are sources if and only if the right external edge of the digon is on the boundary, as shown in the first (i.e.\
leftmost) picture of Figure~\ref{configurations}.
We need not consider the case where both external edges of the digon piece are on the boundary, because Theorem~\ref{D surface} applies to a once-punctured digon.

In the monogon piece, both the arcs in the left pair are never sources, and the arcs of the right pair are sources if and only if the external edge of the monogon is on the boundary.
However, we don't need to consider that case because the surface is a twice-punctured monogon, and Theorem~\ref{D surface} applies.

If $\alpha$ is the external edge of a monogon piece, then each of the two left internal arcs $\gamma$ has $b_{\alpha\gamma}=-1$, so $\alpha$ is not a source.
It remains, then, to consider how external edges of triangle and digon pieces can be sources.
We need to consider two cases.

Suppose $\alpha$ is an edge in a triangle piece and suppose $\gamma$ is the edge reached from $\alpha$ by traversing the boundary of the triangle in a counter-clockwise direction.
If $\gamma$ is not on the boundary, then the triangle contributes $-1$ to $b_{\alpha\gamma}$, so $\alpha$ cannot be a source unless either $\gamma$ is on the boundary or $\alpha$ and $\gamma$ are also in a second triangle that contributes $1$ to $b_{\alpha\gamma}$.

Next suppose $\alpha$ is an external edge in a digon piece.
If $\alpha$ is the left edge, then each of the two internal arcs $\gamma$ has $b_{\alpha\gamma}=-1$, so $\alpha$ is not a source.
If $\alpha$ is the right edge, let $\gamma$ be the left edge.
As in the triangle case, $\alpha$ cannot be a source unless either $\gamma$ is on the boundary or $\alpha$ and $\gamma$ are also in a second triangle that contributes $1$ to $b_{\alpha\gamma}$.

Putting all of these observations together, we see that we must consider three more possibilities obtained by gluing a triangle or digon piece to another triangle or digon piece.
We can glue two triangle pieces together along one edge with opposite edges of the resulting quadrilateral on the boundary as shown in the second picture in Figure~\ref{configurations}.
Conceivably the top and bottom arcs shown in the picture are identified, but we need not consider this case because then the surface is an annulus with two marked points on each boundary component, and Theorem~\ref{D surface} applies.
We can glue two triangle pieces together along two edges, with one of the remaining edges on the boundary as shown in the third picture in Figure~\ref{configurations}.
We can glue a triangle piece along one of its edges to the right edge of a digon piece, with both the left digon edge and the triangle edge counter-clockwise from the glued edge on the boundary, as shown in the fourth and last picture in Figure~\ref{configurations}.
We can glue a triangle piece along two of its edges to the two edges of a digon piece, with the remaining edge of the triangle on the boundary, but we need not consider this case, because the surface is a twice-punctured monogon, and Theorem~\ref{D surface} applies.
We can glue two digon pieces, right edge to right edge, with the remaining two edges on the boundary.
However, we need not consider this case either, because the surface is a twice-punctured digon and Theorem~\ref{D surface} applies.
Finally, we can glue both edges of a digon piece to both edges of another digon piece, but in this case, we obtain a 3-times punctured sphere, which is explicitly disallowed in the definition of marked surfaces \cite[Definition~2.1]{FST}.
Thus the four configurations in Figure~\ref{configurations} are the only local configurations near arcs that are sources, except in surfaces to which Theorem~\ref{D surface} applies.
We will see that the first and third configurations shown are essentially equivalent for our purposes.

Recall from Section~\ref{D surface sec} that $(\alpha|\beta)$ is the sum of four quantities $A$, $B$, $C$, and $D$.
As before, $\alpha_0$ and $\beta_0$ are non-self-intersecting curves homotopic (relative to the set of marked points) to $\alpha$ and $\beta$ respectively, intersecting with each other the minimum possible number of times, transversally each time.
Recall that $B=0$ unless $\alpha_0$ is a loop.
In the configurations of Figure~\ref{configurations}, $\alpha_0$ is never a loop.
Furthermore, the quantity $b_{\alpha\gamma}$ is nonzero only if $\gamma$ is in a triangle with $\alpha$, and none of the arcs making triangles with $\alpha$ is a loop in the configurations of Figure~\ref{configurations}.
Therefore, we can ignore $B$ in all the calculations of intersection numbers in this section.
Recall also that $A$ is the number of intersection points of $\alpha_0$ and $\beta_0$ (excluding intersections at their endpoints), that $C=0$ unless $\alpha_0$ and $\beta_0$ coincide, in which case $C=-1$, and that $D$ is the number of ends of $\beta$ that are incident to an endpoint of $\alpha$ and carry, at that endpoint, a different tag from the tag of $\alpha$ at that endpoint.

We observe that $(\alpha|\beta)$ is invariant under changing all taggings of $\alpha$ and of $\beta$ at some puncture.
Thus for the first (leftmost) picture in Figure~\ref{configurations}, we may as well take $\alpha$ to be the arc tagged notched at the puncture.
Figure~\ref{labeled configurations} shows the configurations of Figure~\ref{configurations} with some additional information.
First, the arc $\alpha'$, obtained by flipping $\alpha$, is shown and labeled.
Also, the arcs $\gamma$ such that $b_{\alpha\gamma}>0$ are labeled.
There is either one arc $\gamma_1$, two arcs $\gamma_1$ and $\gamma_2$, or three arcs $\gamma_1$, $\gamma_2$ and~$\gamma_3$.
The pictures in Figure~\ref{labeled configurations} are re-ordered in the order we will consider them.
We have also redrawn the last configuration more symmetrically.
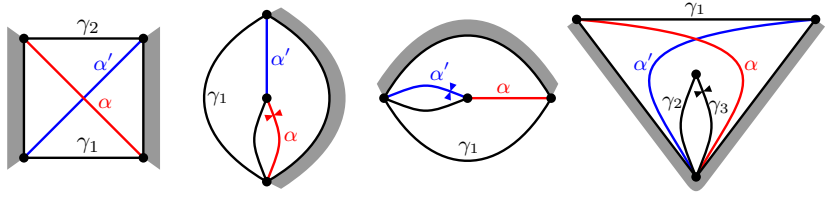
\begin{figure}
  \scalebox{0.9}{\includegraphics{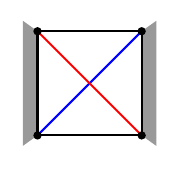}%
  \begin{picture}
    (0,0)(43,-43)
    \put(6,-4){\small\textcolor{red}{$\alpha$}}
    \put(4,12){\small\textcolor{blue}{$\alpha'$}}
    \put(-3,-21){$\gamma_1$}\put(-3,29){$\gamma_2$}
  \end{picture}}
  \scalebox{0.9}{\includegraphics{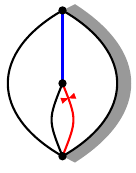}%
  \begin{picture}
    (0,0)(35,-43)
    \put(7,-18){\small\textcolor{red}{$\alpha$}}
    \put(1.5,14){\small\textcolor{blue}{$\alpha'$}}
    \put(-25,-1){\small$\gamma_1$}
  \end{picture}}
  \,\,
  \scalebox{0.9}{\includegraphics{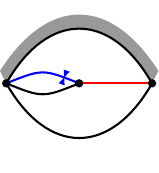}%
  \begin{picture}
    (0,0)(38,-43)
    \put(12,2.5){\small\textcolor{red}{$\alpha$}}
    \put(-16,8){\small\textcolor{blue}{$\alpha'$}}
    \put(-4,-22){\small$\gamma_1$}
  \end{picture}}
  \scalebox{0.9}{\includegraphics{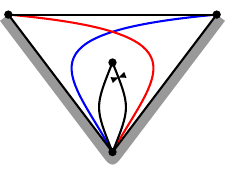}%
  \begin{picture}
    (0,0)(54,-43)
    \put(20,12){\small\textcolor{red}{$\alpha$}}
    \put(-26.5,12){\small\textcolor{blue}{$\alpha'$}}
    \put(-5,36.5){\small$\gamma_1$}
    \put(-14.5,-3){\small$\gamma_2$}
    \put(5,-5){\small$\gamma_3$}
  \end{picture}}
  \caption{Possible local configurations, with more information}
  \label{labeled configurations}
\end{figure}

Our task is simplified by several symmetries.
We have already used the symmetry of changing taggings at a puncture.
Also, any symmetry of a configuration that fixed $\alpha$ and $\alpha'$ or switches $\alpha$ and $\alpha'$ preserves equation~\eqref{D backwards source-sink}.
If the symmetry is orientation-reversing, the absolute value operation in~\eqref{D backwards source-sink} is crucial to the symmetry.
(Note however that this absolute value has been omitted in~\eqref{the task} because we took $\alpha$ to be a source, not a sink.)

We first consider the left picture in Figure~\ref{labeled configurations}.
Since each marked point is on the boundary, there are no relevant taggings.
Contributions to $(\alpha|\beta)$, $(\alpha'|\beta)$, $(\gamma_1|\beta)$, and $(\gamma_2|\beta)$ occur only when $\beta$ intersects the interior of the quadrilateral.
While $\beta$ may intersect the interior of the quadrilateral a number of times, each intersection can be treated separately.
In such an intersection, $\beta$ may either pass through the quadrilateral, terminate at a vertex of the quadrilateral, or connect two vertices of the quadrilateral.
Up to symmetry, as discussed above, there are only three possibilities.
(The relevant symmetry group is the order-$4$ dihedral symmetry group of the rectangle shown.)
Figure~\ref{squarecase} shows the possible intersections of $\beta$ (shown as a dotted line) with the quadrilateral, along with the contributions to $(\alpha|\beta)$, $(\alpha'|\beta)$, $(\gamma_1|\beta)$, and $(\gamma_2|\beta)$.
Each of these is given in the form $A+C+D$.
The quantities $b_{\alpha\gamma_1}$ and $b_{\alpha\gamma_2}$ are both $1$.
In every case, we see that $(\alpha|\beta)=-(\alpha'|\beta)+(\gamma_1|\beta)+(\gamma_2|\beta)$, and therefore \eqref{the task} holds.

\begin{figure}
  $\begin{array}{c|c|c|c|c|}
    &(\alpha|\beta)&(\alpha'|\beta)&(\gamma_1|\beta)&(\gamma_2|\beta)\\\hline
    \raisebox{-15pt}{\includegraphics{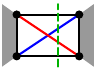}}&1+0+0&1+0+0&1+0+0&1+0+0\\\hline
    \raisebox{-15pt}{\includegraphics{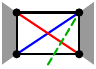}}&1+0+0&0+0+0&1+0+0&0+0+0\\\hline
    \raisebox{-15pt}{\includegraphics{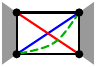}}&1+0+0&0+(-1)+0&0+0+0&0+0+0\\\hline
  \end{array}$
  \caption{}
  \label{squarecase}
\end{figure}

Notice that the second and third pictures of Figure~\ref{labeled configurations} are related by a reflection that switches $\alpha$ with $\alpha'$.
By the symmetry discussed above, we need only consider one of these configurations; we will work with the third picture.
Contributions to $(\alpha|\beta)$, $(\alpha'|\beta)$, and $(\gamma_1|\beta)$ only occur when $\beta$ intersects the digon, and again, we can treat each intersection separately.
Figure~\ref{twotricase} shows all but four of the possible intersections of $\beta$ with the configuration, and shows $(\alpha|\beta)$, $(\alpha'|\beta)$, and $(\gamma_1|\beta)$, again in the form $A+C+D$.
\begin{figure}
  $\begin{array}{c|c|c|c|}
    &(\alpha|\beta)&(\alpha'|\beta)&(\gamma_1|\beta)\\\hline
    \raisebox{-16pt}{\includegraphics{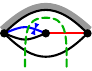}}&1+0+0&1+0+0&2+0+0\\\hline
    \raisebox{-16pt}{\includegraphics{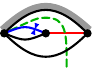}}&1+0+0&0+0+0&1+0+0\\\hline
    \raisebox{-16pt}{\includegraphics{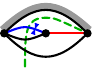}}&0+0+0&1+0+0&1+0+0\\\hline
    \raisebox{-16pt}{\includegraphics{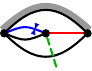}}&0+0+0&0+0+1&1+0+0\\\hline
    \raisebox{-16pt}{\includegraphics{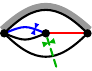}}&0+0+1&0+0+0&1+0+0\\\hline
  \end{array}$
  \caption{}
  \label{twotricase}
\end{figure}
Since $b_{\alpha\gamma_1}=1$, the desired relation is $(\alpha|\beta)=-(\alpha'|\beta)+(\gamma_1|\beta)$, and we see that this relation holds in every case.
Not pictured in Figure~\ref{twotricase} are the four cases where $\beta_0$ coincides with $\alpha_0$ or $\alpha'_0$, with two possible taggings at the point in the center of the digon.
In each of these cases, $(\gamma_1|\beta)=0$ and the $A$ terms of $(\alpha|\beta)$ and $(\alpha'|\beta)$ are both zero.
The other terms are also zero, except that one of $(\alpha|\beta)$ and $(\alpha'|\beta)$ has $C=-1$ and one of $(\alpha|\beta)$ and $(\alpha'|\beta)$ has $D=1$.

Finally, we consider the last picture in Figure~\ref{labeled configurations}.
The two cases where $\beta_0$ coincides with $\alpha_0$ or $\alpha_0'$ are handled analogously to the last case of the quadrilateral condition.
Up to symmetry, there are six remaining cases, pictured in Figure~\ref{digontricase}.
Once again, $b_{\alpha\gamma_i}=1$ for $i\in\set{1,2,3}$, and the desired relation $(\alpha|\beta)=-(\alpha'|\beta)+(\gamma_1|\beta)+(\gamma_2|\beta)+(\gamma_3|\beta)$ holds in every case.
\begin{figure}
  $\begin{array}{c|c|c|c|c|c|}
    &(\alpha|\beta)&(\alpha'|\beta)&(\gamma_1|\beta)&(\gamma_2|\beta)&(\gamma_3|\beta)\\\hline
    \raisebox{-16pt}{\includegraphics{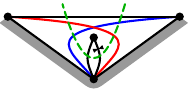}}&2+0+0&2+0+0&2+0+0&1+0+0&1+0+0\\\hline
    \raisebox{-16pt}{\includegraphics{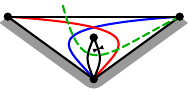}}&2+0+0&1+0+0&1+0+0&1+0+0&1+0+0\\\hline
    \raisebox{-16pt}{\includegraphics{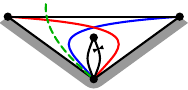}}&1+0+0&0+0+0&1+0+0&0+0+0&0+0+0\\\hline
    \raisebox{-16pt}{\includegraphics{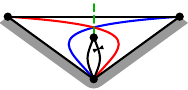}}&1+0+0&1+0+0&1+0+0&0+0+0&0+0+1\\\hline
    \raisebox{-16pt}{\includegraphics{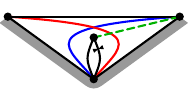}}&1+0+0&0+0+0&0+0+0&0+0+0&0+0+1\\\hline
    \raisebox{-16pt}{\includegraphics{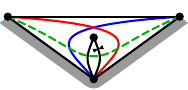}}&1+0+0&1+0+0&0+0+0&1+0+0&1+0+0\\\hline
  \end{array}$
  \caption{}
  \label{digontricase}
\end{figure}

\end{document}